\documentclass[twoside, 12pt]{article}
\usepackage{mathrsfs}
\usepackage{amssymb, amsmath, mathrsfs, amsthm}
\usepackage{graphicx}
\usepackage{color}
\usepackage[top=2cm, bottom=2cm, left=2cm, right=2cm]{geometry}
\usepackage{float, caption, subcaption}

\input{epsf}

\newcommand{\bd}{\begin{description}}
\newcommand{\ed}{\end{description}}
\newcommand{\bi}{\begin{itemize}}
\newcommand{\ei}{\end{itemize}}
\newcommand{\be}{\begin{enumerate}}
\newcommand{\ee}{\end{enumerate}}
\newcommand{\beq}{\begin{equation}}
\newcommand{\eeq}{\end{equation}}
\newcommand{\beqs}{\begin{eqnarray*}}
\newcommand{\eeqs}{\end{eqnarray*}}

\definecolor{DarkGreen}{rgb}{0.2, 0.6, 0.3}

\newtheorem{theorem}{Theorem}[section]

\newtheorem{lemma}{Lemma}[section]

\newtheorem{corollary}[theorem]{Corollary}

\begin{document}
\title{\textbf{Bounds for Gallai-Ramsey functions and numbers} \footnote{Supported by the National
Science Foundation of China (Nos. 11601254, 11551001, 11161037,
61763041, 11661068, and 11461054) and the Qinghai Key Laboratory of
Internet of Things Project (2017-ZJ-Y21).} }

\author{Zhao Wang\footnote{College of Science, China Jiliang University,
Hangzhou 310018, China. {\tt wangzhao@mail.bnu.edu.cn}}, \ \ Yaping
Mao\footnote{Corresponding author} \footnote{School of Mathematics
and Statistis, Qinghai Normal University, Xining, Qinghai 810008,
China. {\tt maoyaping@ymail.com}} \footnote{Academy of Plateau
Science and Sustainability, Xining, Qinghai 810008, China.}, \ \ Ran
Gu \footnote{College of Science, Hohai University, Nanjing, Jiangsu
Province 210098, China. {\tt rangu@hhu.edu.cn}}, \ \ Suping Cui
\footnote{School of Mathematics and Statistis, Qinghai Normal
University, Xining, Qinghai 810008, China. {\tt jiayoucui@163.com}},
\ \ Hengzhe Li \footnote{College of Mathematics and Information
Science, Henan Normal University, Xinxiang 453007, P.R. China. {\tt
lhz@htu.cn}}}
\date{}
\maketitle

\begin{abstract}
For two graphs $G,H$ and a positive integer $k$, the
\emph{Gallai-Ramsey number} $\operatorname{gr}_k(G,H)$ is defined as
the minimum number of vertices $n$ such that any $k$-edge-coloring
of $K_n$ contains either a rainbow (all different colored) copy of
$G$ or a monochromatic copy of $H$. If $G$ and $H$ are both complete
graphs, then we call it Gallai-Ramsey function. Fox and Sudakov
proved $\operatorname{gr}_k(K_s,K_t)\leq s^{4kt}$. Alon et al.
showed that $\operatorname{gr}_k(K_s,K_t)\leq (2s^3+4s^2)^{kt}$. In
this paper, we prove that $\operatorname{gr}_k(K_s,K_t)\leq
2^{kt}s^{3kt}$ for $t\geq 47$. We also give better upper bounds for
$\operatorname{gr}_k(G,H)$ when $G,H$ are some special graphs. In
this paper, we derive some lower bounds for
Gallai-Ramsey functions and numbers by Lov\'{a}sz Local Lemma.\\[2mm]
\end{abstract}

\section{Research Background}

Ramsey theory, named from Frank P. Ramsey, is a branch of
mathematics that studies the conditions under which order must
appear. Problems in Ramsey theory typically ask a question of the
form: ``how many elements of some structure must there be to
guarantee that a particular property will hold?'' More specifically,
Ron Graham describes Ramsey theory as a ``branch of combinatorics''.
We refer the readers to \cite{GrahamRothschildSpencer} for a
classical book of Ramsey theory.

\subsection{Ramsey theorem}

Let $V(G)$, $E(G)$, $e(G)$, $\delta(G)$ be the vertex set, edge set,
size, minimum degree of graph $G$, respectively. An $r$-coloring is
\emph{exact} if all colors are used at least once. In this work, we
consider only edge-colorings of graphs. A coloring of a graph is
called \emph{rainbow} if no two edges have the same color. Let
$[n]=\{1,2,\ldots,n\}$ and $[n]^2=\{Y:Y\subset \{1,2,\ldots,n\}, \
|Y|=2\}$.

We write $n\rightarrow (\ell_1,\ell_2,\ldots, \ell_r)$ if, for every
$r$-coloring of $[n]^2$, there exists $i$, $1\leq i\leq r$, and a
set $T$, $|T|=\ell_i$ so that $[T]^2$ is colored $i$. The
\emph{Ramsey function} $R(\ell_1,\ell_2,\ldots, \ell_r)$ denotes the
minimal $n$ such that
$$
n\rightarrow (\ell_1,\ell_2,\ldots, \ell_r).
$$

\begin{theorem}{\upshape \textbf{(Ramsey's theorem}) \cite{GrahamRothschildSpencer}}  \label{Ramsey}
The function $R$ is well defined, that is, for all
$\ell_1,\ell_2,\ldots, \ell_r$ there exists $n$ such that
$$
n\rightarrow (\ell_1,\ell_2,\ldots, \ell_r).
$$
\end{theorem}

The Ramsey function only consider complete graphs. But later, the
Ramsey number are considered for general graphs. Given $k$ graphs
$H_1,H_2,\ldots,H_k$, let ${\rm R}(H_1,H_2,\ldots,H_k)$ denote the
minimum number of vertices $n$ needed so that every
$k$-edge-coloring of $K_{n}$ contains a monochromatic $H_i$, where
$1\leq i\leq n$. If $H_1=H_2=\cdots=H_k$, then we write the number
as ${\rm R}_k(H)$.

We refer the readers to \cite{MR1670625} for a dynamic survey of
small Ramsey numbers.

\subsection{Gallai-Ramsey number and function}

Colorings of complete graphs that contain no rainbow triangle have
very interesting and somewhat surprising structure. In 1967, Gallai
\cite{MR0221974} first examined this structure under the guise of
transitive orientations. The result was reproven in \cite{MR2063371}
in the terminology of graphs and can also be traced to
\cite{MR1464337}. For the following statement, a trivial partition
is a partition into only one part.

\begin{theorem}[\cite{MR1464337,MR0221974,MR2063371}]\label{Thm:G-Part}
In any coloring of a complete graph containing no rainbow triangle,
there exists a nontrivial partition of the vertices (that is, with
at least two parts) such that there are at most two colors on the
edges between the parts and only one color on the edges between each
pair of parts.
\end{theorem}

For ease of notation, we refer to a colored complete graph with no
rainbow triangle as a \emph{Gallai-coloring} and the partition
provided by Theorem~\ref{Thm:G-Part} as a \emph{Gallai-partition}.
The induced subgraph of a Gallai colored complete graph constructed
by selecting a single vertex from each part of a Gallai partition is
called the \emph{reduced graph} of that partition. By
Theorem~\ref{Thm:G-Part}, the reduced graph is a $2$-colored
complete graph.

Although the reduced graph of a Gallai partition uses only two
colors, the original Gallai-colored complete graph could certainly
use more colors. With this in mind, we consider the following
generalization of the Ramsey numbers. Given two graphs $G$ and $H$,
the \emph{$k$-colored Gallai-Ramsey number}
$\operatorname{gr}_k(G:H)$ is defined to be the minimum integer $m$
such that every $k$-edge-coloring of the complete graph on $m$
vertices contains either a rainbow copy of $G$ or a monochromatic
copy of $H$. With the additional restriction of forbidding the
rainbow copy of $G$, it is clear that $\operatorname{gr}_k(G:H)\leq
{\rm R}_k(H)$ for any graph $G$.

In \cite{GSSS10}, Gy{\'a}rf{\'a}s et al. obtained the following nice
result.
\begin{theorem}{\upshape \cite{GSSS10}}
Let $H$ be a fixed graph with no isolated vertices. If $H$ is not
bipartite, then $\operatorname{gr}_{k}(K_{3} : H)$ is exponential in
$k$. If $H$ is bipartite, then $\operatorname{gr}_{k}(K_{3} : H)$ is
linear in $k$.
\end{theorem}

We refer the interested reader to \cite{MagnantNowbandegani} for a
book and \cite{ChenLiPei, FMO14, FoxSudakov2008, FujitaMagnant,
GSSS10, LiWang, LiBesseMagnantWangWatts, MaoWangMaganantSchiermeyer,
WangMaoZouMaganant, ZhaoWei} for recent papers on Gallai-Ramsey
numbers.

We write $n\overset{\operatorname{gr}_{k}}{\longrightarrow} (s,t)$
if, for every $k$-coloring of $[n]^2$, there exists a set $S$ such
that $[S]^2$ is rainbow or there exists a set $T$ so that $[T]^2$ is
monochromatic. If $|S|=s$ and $|T|=t$, then the \emph{Gallai-Ramsey
function} $\operatorname{gr}_k(s,t,t,\ldots,t)$ ($k$ times of $t$)
or $\operatorname{gr}_k(s,t)$ denotes the minimal $n$ such that
$$
n\overset{\operatorname{gr}_k}{\longrightarrow} (s,t).
$$
Note that $\operatorname{gr}_k(s,t)=\operatorname{gr}_k(K_s:K_t)$.
If $k=2$, then
$\operatorname{gr}_2(s,t)=\operatorname{gr}_2(K_s:K_t)=\operatorname{r}(s,t)$.

Since $\operatorname{gr}_k(G:H)\leq {\rm R}_k(H)$ for any graph $G$,
it follows that the following corollary is immediate from Theorem
\ref{Ramsey}.
\begin{corollary}
The function $\operatorname{gr}_k$ is well defined, that is, for all
$s,t$ there exists $n$ such that
$$
n\overset{\operatorname{gr}_k}{\longrightarrow} (s,t).
$$
\end{corollary}

\subsection{Main results}

In 1981, Erd\"{o}s \cite{Erdos} proposed studying the following
generalization of the classical Ramsey problem. Let $p,q$ be
positive integers with $2\leq q\leq {p\choose 2}$. A
\emph{$(p,q)$-coloring} of $K_n$ is an edge-coloring such that every
copy of $K_p$ receives at least $q$ distinct colors. Let $f(n,p,q)$
be the minimum number of colors in a $(p,q)$-coloring of $K_n$.
Determining the numbers $f(n,p,2)$ is equivalent to determining the
multicolor Ramsey numbers $\operatorname{R}_k(K_p)$, as an
edge-coloring is a $(p,2)$-coloring if and only if it does not
contain a monochromatic $K_p$. Let $g(k)$ be the largest positive
integer $n$ for which there is a $k$-edge-coloring of $K_n$, in
which every $K_4$ receives at least three colors, i.e., for which
$f(n, 4, 3) \leq k$. Restated, $g(k) + 1$ is the smallest positive
integer $n$ for which every $k$-edge-coloring of the edges of $K_n$
contains a $K_4$ that receives at most two colors.

Fox and Sudakov \cite{FoxSudakov2008} obtained the first exponential
upper bound for this problem.
\begin{theorem} {\upshape \cite{FoxSudakov2008}}\label{th-FoxSudakov2008}
For $k > 2^{100}$, we have $g(k) < 2^{2000k}$.
\end{theorem}

They got the following upper bound.
\begin{theorem} {\upshape \cite{FoxSudakov2008}}
$\operatorname{gr}_k(K_s,K_t)\leq s^{4kt}$.
\end{theorem}

Given a positive integer $m$, we define an edge-coloring of a (host)
graph to be \emph{$m$-good} if each color appears at most $m$ times
at each vertex. Given any graph $H$, let $g(m, H)$ denote the
smallest $n$ such that every $m$-good coloring of $E(K_n)$ yields a
rainbow copy of $H$.

Alon et al. \cite{AlonJiangMillerPritikin} derived the following
result.
\begin{theorem}{\upshape \cite{AlonJiangMillerPritikin}}\label{th1-7}
For all positive integers $m,s$ we have $g(m,K_s)\leq 2ms^3+4ms^2$.
\end{theorem}

\begin{corollary}\label{cor1-8}
$\operatorname{gr}_k(K_s,K_t)\leq (2s^3+4s^2)^{kt}$.
\end{corollary}

In Section $2$, we get the following result by the idea in
\cite{AlonJiangMillerPritikin, FoxSudakov2008}.
\begin{theorem}\label{th-main}
Let $r,s,t$ be three integers with $r\leq 2(s-r)^3-{s-r\choose 2}$
and $t\geq 47$. Let $G_i \ (1\leq i\leq a)$ be a graph obtained from
the following steps.
\begin{itemize}
\item[] $(1)$ $G_1$ is a graph of order $s$;

\item[] $(2)$ $G_{i}$ is a graph obtained
from $G_{i-1}$ by deleting all leaves;

\item[] $(3)$ There exists an integer $a$ such that $G_{a}$ is a complete graph.
\end{itemize}
If the total number of deleted edges is $r$, then
$$
\operatorname{gr}_k(G_1,K_t)\leq 2^{kt}(s-r)^{3kt}.
$$
\end{theorem}

The following corollary is immediate.
\begin{corollary}\label{cor1-9}
For $t\geq 47$, $\operatorname{gr}_k(K_s,K_t)\leq 2^{kt}s^{3kt}$.
\end{corollary}

The following result is also due to Alon et al.
\cite{AlonJiangMillerPritikin}.
\begin{theorem}{\upshape \cite{AlonJiangMillerPritikin}}\label{th-Alon}
For any positive integer $s$ sufficiently large, there exists an
absolute constant $c$ such that, for all admissible integers $m$ and
$s$, we have $g(m, K_s)<\frac{cms^3}{\ln s}$.
\end{theorem}

We can give a better upper bound by Theorem \ref{th-Alon}.
\begin{corollary}\label{cor1-11}
For any positive integer $s$ sufficiently large, there exists an
absolute constant $c$ such that, for all admissible integers $m$ and
$s$, we have $\operatorname{gr}_k(K_s,K_t)\leq (\frac{cs^3}{\ln
s})^{kt}$.
\end{corollary}

\begin{theorem}{\upshape \cite{AlonJiangMillerPritikin}}
Let $G$ be a graph with $s$ vertices and maximum degree $d$. For all
positive integers $m$, we have $g(m, G)\leq 2md^2s+32md^4+4s$.
\end{theorem}

\begin{theorem}\label{th-F}
Let $s,t$ be two integers, and let $F_i \ (1\leq i\leq a)$ be a
graph obtained from the following steps.
\begin{itemize}
\item[] $(1)$ $F_1$ is a graph of order $s$;

\item[] $(2)$ $F_{i}$ is a graph obtained
from $F_{i-1}$ by deleting all leaves;

\item[] $(3)$ There exists an integer $a$ such that $F_{a}$ has no leaves.
\end{itemize}
If $\Delta(F_{a})=d\geq 2$, and $e(F_1)=b$ , $b\leq
2d^2(s-r)+32d^4$, and the total number of deleted edges is $r$, then
$$
\operatorname{gr}_k(F_1,K_t)\leq
(2d^2(s-r)+32d^4)^{kt}+4(s-r)\frac{(2d^2(s-r)+32d^4)^{kt}-1}{2d^2(s-r)+32d^4-1}.
$$
\end{theorem}

\begin{theorem}\label{th-tree}
Let $s,t$ be two integers, and let $T$ be a tree of order $s$. Then
$$
\operatorname{gr}_k(T,K_t)\leq (s-1)^{kt}.
$$
\end{theorem}

The probabilistic method is a powerful technique for approaching
asymptotic combinatorial problems. The following probability result,
due to L. Lov\'{o}sz, fundamentally improves the Existence argument
in many instances.

Let $A_1,\ldots,A_n$ be events in a probability space $\Omega$. We
say that the graph $\Gamma$ with vertex set $\{1,2,\ldots,n\}$ is a
\emph{dependency graph} $\{A_1,A_2,\ldots,A_n\}$ if:
\begin{center}
$\{i\}$ not joined to $j_1,j_2,\ldots,j_s$ $\Longrightarrow $ $A_i$
and $A_{j_1}\cap A_{j_2}\cap \cdots \cap A_{j_s}$ are independent.
\end{center}

\begin{theorem} {\upshape (\textbf{Lov\'{a}sz Local Lemma} \cite{ErdosLovasz})}\label{th-LLL}
Let $A_1,\ldots,A_n$ be events in a probability space $\Omega$ with
dependence graph $\Gamma$. Suppose that there exists
$x_1,\ldots,x_n$ with $0<x_i\leq 1$ such that
$$
\Pr[A_i]<(1-x_i)\prod_{\{i,j\}\in \Gamma}x_j, \ 1\leq i\leq n.
$$
Then $\Pr[\bigwedge_{i} \overline{A_i}]>0$.
\end{theorem}

A slightly more convenient form of the local lemma results from the
following observation. Set
$$
y_i=\frac{1-x_i}{x\Pr[A_i]},
$$
so that
$$
x_i=\frac{1}{1+y_i\Pr[A_i]}.
$$

Since $1+z\leq \exp(z)$, we have:

\begin{corollary}{\upshape \cite{GrahamRothschildSpencer}}\label{Lemma-LLL}
Suppose that $A_1,\ldots,A_n$ are events in a probability space
having dependence graph $\Omega$, and there exist positive
$y_1,y_2,\ldots,y_n$ satisfying
$$
\ln y_i>\sum_{\{i,j\}\in \Gamma}y_j\Pr[A_j]+y_i\Pr[A_i],
$$
for $1\leq i\leq n$. Then $\Pr[\bigwedge \overline{A_i}]>0$.
\end{corollary}

In \cite{Spencer}, Spencer studied the some asymptotic lower bounds
for Ramsey functions. Li et al. \cite{LiRousseauZang} investigated
the asymptotic upper bounds for Ramsey functions. Chen et al.
\cite{ChenHattinghRousseau} got the asymptotic bounds for
irredundant and mixed Ramsey numbers. Caro et al.
\cite{CaroLiRousseauZhang} obtained the asymptotic bounds for some
bipartite graphs. In \cite{GodboleSkipperSunley}, Godbole et al.
studied the asymptotic lower bound on the diagonal Ramsey numbers.
Erd\"{o}s and Hattingh \cite{ErdosHattingh} investigated the
asymptotic bounds for irredundant Ramsey numbers.

In Subsection 2.1, we obtain a lower bound of Gallai-Ramsey function
for a fixed probability of receiving colors for each edge. By
Lov\'{o}sz Local Lemma, we derive another lower bound of
Gallai-Ramsey function for a flexible probability of receiving
colors for each edge. In Subsection 2.2, we got some lower bounds
for Gallai-Ramsey numbers by the same method.

\section{Upper bounds}

For an edge-coloring of $K_n$, a vertex $x$, and a color $i$, let
$d_i(x)$ denote the degree of vertex $x$ in color $i$. Our first
lemma shows that if, for every vertex $x$ and color $i$, $d_i(x)$ is
not too large, then the coloring contains many rainbow cliques.

\begin{lemma}{\upshape \cite{FoxSudakov2008}}\label{lem-FoxSudakov2008}
If an edge-coloring of the complete graph $K_n$ satisfies
$d_i(x)\leq \delta n$ for each $x\in V(K_n)$ and each color $i$,
then this coloring has at most $\frac{5}{8}\delta t^4{n\choose t}$
non-rainbow copies of $K_t$.
\end{lemma}

\begin{lemma}\label{lem-FoxSudakov2008}\label{lem2-2}
Let $t$ be an integer with $t\geq 47$. If an edge-coloring of the
complete graph $K_n$ satisfies $d_i(x)\leq \frac{n}{2t^3}$ for each
$x\in V(K_n)$ and each color $i$, then there is a rainbow copy of
$K_t$.
\end{lemma}
\begin{proof}
If a $K_t$ is not rainbow, then it has two adjacent edges of the
same color or two nonadjacent edges of the same color. Let
$\nu(i,t,n)$ be the number of copies of $K_t$ in $K_n$ in which
there are at least two adjacent edges of color $i$. To bound the
number of such $K_t$ we can first choose the vertex, then the two
edges with color $i$ incident to this vertex and then the remaining
$t-3$ vertices. Hence, the number of $K_t$'s for which there is a
vertex with degree at least two in some color is at most
\begin{eqnarray*}
\sum_{i}\nu(i,t,n)&\leq&\sum_{i}\sum_{x\in V}{d_i(x)\choose 2}{n-3\choose t-3}\leq 2nt^3{\frac{n}{2t^3}\choose 2}{n-3\choose t-3}\\[0.2cm]
&<& \frac{n^3}{4t^3}\left(\frac{t}{n}\right)^3{n\choose
t}=\frac{1}{4}{n\choose t}.
\end{eqnarray*}

Let $\psi(i,t,n)$ be the number of copies of $K_t$ in $K_n$ in which
there is a matching of size at least two in color $i$. Let $x$ be
the number of two edges of a matching with the same color. Then the
number of $K_t$'s in which there is a matching of size at least two
in some color is at most
\begin{eqnarray*}
\sum_{i}\psi(i,t,n)&\leq& x{n-4\choose t-4}-\sum^{\lfloor \frac{t}{4}\rfloor}_{i=2}{x\choose i}{n-4i\choose t-4i}\\[0.2cm]
&\leq& x{n-4\choose t-4}-{x\choose 2}{n-8\choose t-8}\\[0.2cm]
&\leq& \frac{{n-4\choose t-4}}{{n-8\choose t-8}}{n-4\choose t-4}-{\frac{{n-4\choose t-4}}{{n-8\choose t-8}}\choose 2}{n-8\choose t-8}\\[0.2cm]
&=&
\frac{1}{2}\left(\frac{(n-4)(n-5)(n-6)(n-7)}{(t-4)(t-5)(t-6)(t-7)}+1\right){n-4\choose
t-4}\leq \frac{3}{4}{n\choose t}.
\end{eqnarray*}

Hence, the number of $K_t$'s which are not rainbow is at most
$\frac{1}{4}{n\choose t}+\frac{3}{4}{n\choose t}-1={n\choose t}-1$,
completing the proof.
\end{proof}

The following idea for the upper bound is from
\cite{FoxSudakov2008}.
\begin{lemma}\label{lem2-3}
Let $r,s,t$ be three integers with $r\leq 2(s-r)^3-{s-r\choose 2}$
and $t\geq 47$. Let $G_i \ (1\leq i\leq a)$ be a graph obtained from
the following steps.
\begin{itemize}
\item[] $(1)$ $G_1$ is a graph of order $s$;

\item[] $(2)$ $G_{i}$ is a graph obtained
from $G_{i-1}$ by deleting all leaves;

\item[] $(3)$ There exists an integer $a$ such that $G_{a}$ is a complete graph.
\end{itemize}
If the total number of deleted edges is $r$, and an edge-coloring of
$K_n$ satisfies $d_i(x)\leq \frac{n}{2(s-r)^3}$ for each $x\in
V(K_n)$ and each color $i$, then there is a rainbow copy of $G_1$ in
$K_n$.
\end{lemma}

\begin{proof}
By Lemma \ref{lem2-2}, for every edge-coloring of $K_n$, there
exists a rainbow $K_{s-r}$. Since $d_i(x)\leq \frac{n}{2(s-r)^3}$
for each $x\in V(K_n)$ and each color $i$, it follows that for any
vertex $u\in K_{s-r}$, the number of different colors of $u$ at
least $2(s-r)^3$. So there is a rainbow copy of $G$.
\end{proof}

Let $M_s(t_1,\ldots,t_k)$ be the maximum $n$ such that there is a
$k$-edge-coloring of $K_n$ with colors $\{1,\ldots,k\}$ without a
rainbow $G$ defined in Lemma \ref{lem2-3} and without a
monochromatic $K_{t_i}$ in color $i$ for $1\leq i\leq k$.

\begin{lemma}\label{lem2-4}
$$
M_s(t_1,\ldots,t_k)\leq 2(s-r)^3 \max_{1\leq i\leq k}
M_s(t_1,\ldots,t_{i-1},\ldots,t_k).
$$
\end{lemma}
\begin{proof}
By Lemma \ref{lem2-3}, for every edge-coloring of $K_n$ without a
rainbow $G$, there is a vertex $v$ with degree at least $n/2(s-r)^3$
in some color $i$. If the coloring of $K_n$ does not contain a
monochromatic $K_{t_i}$ in color $i$, then the neighborhood of $v$
in color $i$ has at least $n/2(s-r)^3$ vertices and does not contain
$K_{t_i}$ in color $i$, completing the proof.
\end{proof}

\noindent \textbf{Proof of Corollary \ref{cor1-8}:} Let
$m=\frac{n}{2s^3+4s^2}$ in Theorem \ref{th1-7} and apply Lemma
\ref{lem2-4}, the result is immediate.

\noindent \textbf{Proof of Theorem \ref{th-main}:} From Lemmas
\ref{lem2-2}, \ref{lem2-3} and \ref{lem2-4}, the result follows.

\noindent \textbf{Proof of Corollary \ref{cor1-11}:} Let
$m=\frac{n\ln s}{cs^3}$ in Theorem \ref{th1-7} and apply Lemma
\ref{lem2-4}, the result follows.

\begin{lemma}\label{lem2-5}
Let $r,s,t$ be three integers. Let $F_i \ (1\leq i\leq a)$ be a
graph obtained from the following steps.
\begin{itemize}
\item[] $(1)$ $F_1$ is a graph of order $s$;

\item[] $(2)$ $F_{i}$ is a graph obtained
from $F_{i-1}$ by deleting all leaves;

\item[] $(3)$ There exists an integer $a$ such that $F_{a}$ has no leaves.
\end{itemize}
Let $\Delta(F_{a})=d\geq 2$, and $e(F_1)=b$ , $b\leq
2d^2(s-r)+32d^4$, and the total number of deleted edges is $r$. If
an edge-coloring of $K_n$ satisfies $d_i(x)\leq
\frac{n-4(s-r)}{2d^{2}(s-r)+32d^4}$ for each $x\in V(K_n)$ and each
color $i$, then there is a rainbow copy of $F_1$ in $K_n$.
\end{lemma}
\begin{proof}
Let $m=\frac{n-4(s-r)}{2d^{2}(s-r)+32d^4}$ in Theorem \ref{th1-7}, we have 
$g(\frac{n-4(s-r)}{2d^{2}(s-r)+32d^4}, G)\leq n$. For every
edge-coloring of $K_n$, there exists a rainbow copy of $F_{a}$.
Since $d_i(x)\leq \frac{n-4(s-r)}{2d^{2}(s-r)+32d^4}$ for each $x\in
V(K_n)$ and each color $i$, it follows that for any vertex $u\in
F_{a}$, the number of different colors of $u$ at least
$2d^2(s-r)+32d^4$. So there is a rainbow copy of $F_{1}$.
\end{proof}

\noindent \textbf{Proof of Theorem \ref{th-F}:} From Lemmas
\ref{lem2-4} and \ref{lem2-5}, the result follows.

\begin{lemma}\label{lem2-6}
let $T$ be a tree of order $s$. If an edge-coloring of the $T$
satisfies $d_i(x)\leq \frac{n}{s-1}$ for each $x\in V(T)$ and each
color $i$, then there is a rainbow copy of $T$.
\end{lemma}

\begin{proof}
For any vertex $u\in T$, the number of different colors of $u$ is at
least $s-1$, and so there is a rainbow copy of $T$.
\end{proof}

\noindent \textbf{Proof of Theorem \ref{th-tree}:} From Lemmas
\ref{lem2-4} and \ref{lem2-6}, the result follows.

\section{Lower bounds}

\subsection{Results for Gallai-Ramsey function}

For a fixed probability of receiving colors for each edge, we can
derive the following lower bound of $\operatorname{gr}_k(s,t)$.
\begin{theorem}\label{th-R(k,k)}
Let $s,t$ be two positive integers with $r,s\geq 3$. For $k\geq
{s\choose 2}$, we have
$$
\operatorname{gr}_k(K_s,K_t)>\frac{1}{e}\cdot
\min\left\{\frac{s}{\sqrt[s]{L+k^{1-{t\choose
2}}}},\frac{t}{\sqrt[t]{L+k^{1-{t\choose 2}}}}\right\},
$$

where $L=\left(\frac{e}{{s\choose 2}}\right)^{{s\choose 2}}{s\choose
2}!$.
\end{theorem}
\begin{proof}
More precisely, we show that if
$$
\max\left\{\left(\frac{ne}{s}\right)^s\left(L+k^{1-{t\choose
2}}\right),\left(\frac{ne}{t}\right)^t\left(L+k^{1-{t\choose
2}}\right)\right\}<1,
$$
then $\operatorname{gr}_k(s,t)>n$, that is, there exists a
$k$-coloring of $K_n$ with vertex set $\{u_1,u_2,\ldots,u_n\}$
containing neither a rainbow $K_s$ nor a monochromatic $K_t$.
Consider a random $k$-coloring of $K_n$, where the color of each
edge is determined by the toss of a fair coin. More precisely, we
have a probability space whose elements are the $k$-colorings of
$K_n$, and whose probabilities are determined by setting
\begin{equation}          \label{eq1-1}
\Pr[\{u_i,u_j\} \ is \ c_x]=k^{-1},
\end{equation}
where $u_i,u_j\in V(K_n)$, $c_1,c_2,\ldots,c_k$ are the all colors
and $1\leq x\leq k$, for all $i,j$ and making these probabilities
mutually independent.

Thus there are $k^{{n\choose 2}}$ colorings, each with probability
$k^{-{n\choose 2}}$. For any set of vertices $S$, $|S|=s$, let $A_S$
denote the event ``$S$ is rainbow.'' Then
$$
\Pr[A_S]=\frac{k(k-1)(k-2)\cdots \left(k-{s\choose
2}+1\right)}{k^{{s\choose 2}}}.
$$
For any set of vertices $T$, $|T|=t$, let $B_T$ denote the event
``$T$ is monochromatic.'' Then
$$
\Pr[B_T]=k^{1-{t\choose 2}},
$$
as the ${t\choose 2}$ ``coin flips'' to determine the colors of
$[S]^2$ must be the same.

The event ``some $s$-element set of vertices $S$ is rainbow'' is
represented by $\bigvee_{|S|=s}A_S$, and ``some $t$-element set of
vertices $T$ is monochromatic'' is represented by
$\bigvee_{|T|=t}B_T$. Then
\begin{eqnarray*}
\Pr\left[\left(\bigvee_{|S|=s}A_S\right)\bigvee
\left(\bigvee_{|T|=t}B_T\right)\right]
&\leq&\Pr\left[\left(\bigvee_{|S|=s}A_S\right)\right]+\Pr\left[\left(\bigvee_{|T|=t}B_T\right)\right]\\[0.2cm]
&\leq &\sum_{|S|=s}\Pr[A_S]+\sum_{|T|=t}\Pr[B_T]\\[0.2cm]
&=&{n\choose s}\frac{k(k-1)(k-2)\cdots \left(k-{s\choose
2}+1\right)}{k^{{s\choose 2}}}+{n\choose t}k^{1-{t\choose 2}}\\[0.2cm]
&<&{n\choose s}\binom{k}{{s\choose 2}}\frac{{s\choose
2}!}{k^{{s\choose 2}}}+{n\choose t}k^{1-{t\choose 2}}\\[0.2cm]
&<&{n\choose s}\left(\frac{ke}{{s\choose 2}}\right)^{{s\choose
2}}\frac{{s\choose 2}!}{k^{{s\choose 2}}}+{n\choose t}k^{1-{t\choose
2}}\\[0.2cm]
&<&{n\choose s}\left(\frac{e}{{s\choose 2}}\right)^{{s\choose
2}}{s\choose 2}!+{n\choose t}k^{1-{t\choose 2}}\\[0.2cm]
&=&{n\choose s}L+{n\choose t}k^{1-{t\choose 2}}.
\end{eqnarray*}
Let $N={n\choose s}L+{n\choose t}k^{1-{t\choose 2}}$. If $s\leq t$,
then
$$
N<{n\choose t}\left(L+k^{1-{t\choose 2}}\right)\leq
\left(\frac{ne}{t}\right)^t\left(L+k^{1-{t\choose 2}}\right), \ {\rm
and} \ {\rm hence} \ n<\frac{t}{e\sqrt[t]{L+k^{1-{t\choose 2}}}}.
$$
If $s>t$, then
$$
N<{n\choose s}\left(L+k^{1-{t\choose 2}}\right)\leq
\left(\frac{ne}{s}\right)^s\left(L+k^{1-{t\choose 2}}\right), \ {\rm
and} \ {\rm hence} \ n<\frac{s}{e\sqrt[s]{L+k^{1-{t\choose 2}}}}.
$$
From the above argument, we have
$$
\operatorname{gr}_k(s,t)>\frac{1}{e}\cdot
\min\left\{\frac{s}{\sqrt[s]{L+k^{1-{t\choose
2}}}},\frac{t}{\sqrt[t]{L+k^{1-{t\choose 2}}}}\right\}.
$$
\end{proof}

For a flexible probability of receiving colors for each edge, we
have the following two results.
\begin{theorem}\label{th2-2}
If, for some $p_1,p_2,\ldots,p_k$, $0\leq p_i\leq 1 \ (1\leq i\leq
k)$,
$$
{n\choose
s}\left[\left(\frac{s(s-1)}{2}\right)!\right]\sum_{x_1,x_2,\ldots,x_{{s
\choose 2}}\subseteq \{1,2,\ldots,k\}}p_{x_1}p_{x_2}\cdots p_{{s
\choose 2}}+{n\choose t}\sum_{x_i\in
\{1,2,\ldots,k\}}p_{x_i}^{{t\choose 2}}<1,
$$
then $\operatorname{gr}_k(s,t)>n$, where $k\geq {s \choose 2}$ and
$c_1,c_2,\ldots,c_k$ are all $k$ colors.
\end{theorem}
\begin{proof}
We use the existence argument of Theorem \ref{th-R(k,k)}, replacing
by (\ref{eq1-1}),
$$
\Pr[\{u_i,u_j\} \ is \ c_x]=p_x,
$$
where $u_i,u_j\in V(K_n)$, $c_1,c_2,\ldots,c_k$ are the all colors
and $1\leq x\leq k$, for all $i,j$ and making these probabilities
mutually independent.

For $S$, $|S|=s$ let $A_S$ be the event ``$[S]^2$ is rainbow," and
for $T$, $|T|=t$, let $B_T$ be the event ``$[T]^2$ is
monochromatic." Then
$$
\Pr\left[\left(\bigvee_{|S|=s}A_S\right)\bigvee
\left(\bigvee_{|T|=t}B_T\right)\right]<1,
$$
so the desired coloring of $K_n$ exists.
\end{proof}

\begin{theorem}\label{th2-3}
Let $k,s,t$ be two positive integers with $r,s\geq 6$ and $k\geq
{s\choose 2}$. Then
$$
\operatorname{gr}_k(s,t)>\frac{1}{\beta
c_2}\left(\frac{(t-1)N^{-1/\gamma}}{\ln
((t-1)N^{-1/\gamma})}\right)^{\beta},
$$
where
$$
\beta=\left[{s\choose 2}-1\right]/(s+1), \ \gamma={s\choose 2}-1, \
\ N={s\choose 2}(k-1)^{2-{s\choose 2}}(k-2)(k-3)\cdots
\left(k-{s\choose 2}+1\right).
$$
\end{theorem}

\begin{proof}
Let the edges of $K_n$ be independently $k$-colored with the
probability that an edge is colored $c_i \ (1\leq i\leq k-1)$ always
being $\frac{p}{k-1}$, and $c_k$ being $1-p$. To each $s$-element
subset of vertices $S$ associate the event $A_{S}$ that all the
edges spanned by $S$ have colored rainbow. To each $t$-element
subset of vertices $T$ associate the event $B_{T}$ that all the
edges spanned by $T$ have colored monochromatic. Observe that
$\operatorname{gr}_k(s,t)>n$ if
$$
\Pr\left[\left(\bigwedge_S\overline{A_S}\right)\bigwedge
\left(\bigwedge_T\overline{B_T}\right)\right]>0.
$$

Let $\Gamma$ denote the graph ${n\choose s}+{n\choose t}$ vertices
corresponding to all possible $A_S$ and $B_T$, where $\{A_S,B_T\}$
is an edge of $\Gamma$ if and only if $|S\cap T|\geq 2$ (i.e., the
events $A_S$ and $B_T$ are independent), the same applies to pairs
of the form $\{A_S,A_{S'}\}$ and $\{B_T,B_{T'}\}$. Let $N_{AA}$
denote the number of vertices of the form $A_S$ for some $S$ joined
to some other vertex of this form, and let $N_{AB},N_{BA}$ and
$N_{BB}$ be defined analogously. If there exist positive $p,y,z$
such that
\begin{equation}          \label{eq2}
\log y>y\Pr[A_S](N_{AA}+1)+z\Pr[B_T]N_{AB},\ \ \ \  \log
z>y\Pr[A_S]N_{BA}+z\Pr[B_T](N_{BB}+1),
\end{equation}
then ${\rm GR}_k(s,t)>n$. Since
\begin{eqnarray*}
\Pr[A_S]&=&(k-1)(k-2)\cdots \left(k-{s\choose
2}\right)\left(\frac{p}{k-1}\right)^{s\choose
2}\\[0.2cm]
& &+{s\choose 2}(1-p)(k-1)(k-2)\cdots \left(k-{s\choose
2}+1\right)\left(\frac{p}{k-1}\right)^{{s\choose 2}-1}\\[0.2cm]
&\leq &{s\choose 2}(k-1)(k-2)\cdots \left(k-{s\choose
2}+1\right)\left(\frac{p}{k-1}\right)^{{s\choose
2}-1}\left[\left(k-{s\choose
2}\right)\left(\frac{p}{k-1}\right)+(1-p)\right]\\[0.2cm]
&\leq &Np^{{s\choose 2}-1},
\end{eqnarray*}
and
\begin{eqnarray*}
\Pr[B_T]&=&(1-p)^{{t\choose
2}}+(k-1)\left(\frac{p}{k-1}\right)^{{t\choose 2}}\leq
(1-p)^{{t\choose
2}}+(k-1)\left(\frac{p}{k-1}\right)^{{t\choose 2}}\\[0.2cm]
&\leq&(1-p)^{{t\choose 2}}+p\left(\frac{p}{k-1}\right)^{{t\choose
2}-1}\leq (1-p)^{{t\choose 2}}+p(1-p)^{{t\choose
2}-1}=(1-p)^{{t\choose 2}-1},
\end{eqnarray*}
it follows that
$$
N(A,B)={n\choose t}-{n-s\choose t}-s{n-s-1\choose t-1}\leq {n\choose
t}, \ \ N(A,A)+1\leq {n\choose s},
$$
$$
N(B,B)+1={n\choose t}-{n-t\choose t}-t{n-t\choose t-1}<{n\choose t},
\  N(B,A)\leq {n\choose s}.
$$
Set
$$
p=c_1n^{-1/\beta}\cdot {N}^{(-1)/\gamma}, \ \ \ \
t-1=c_2n^{1/\beta}(\ln n) {N}^{1/\gamma},
$$
and
$$
z=\exp\left[c_3n^{1/\beta}(\ln n)^2{N}^{1/\gamma}\right], \ \ \ \
y=1+\epsilon,
$$
where
$$
\ln (1+\epsilon)>\frac{(1+\epsilon)c_1^{{s\choose
2}-1}}{n}+\exp(n^{1/\beta}(\ln
n)^2N^{1/\gamma})\left(c_3+c_2-\frac{c_1c_2^2}{4}\right).
$$

Observe that
$$
\ln y>y\cdot Np^{{s\choose 2}-1}\cdot {n\choose
s}+z\left[(1-p)^{{t\choose 2}-1}\right]\cdot {n\choose t}
$$
and
$$
\ln z>y\cdot Np^{{s\choose 2}-1}\cdot {n\choose s}+
z\left[(1-p)^{{t\choose 2}-1}\right]\cdot {n\choose t}.
$$

If $c_3+c_2-\frac{c_1c_2^2}{4}<0$ and $t$ is large, then the
equations (\ref{eq2}) hold. Since
$$
t-1=c_2n^{1/\beta}(\ln n) {N}^{1/\gamma}=(\beta c_2)n^{1/\beta}(\ln
n^{1/\beta}) {N}^{1/\gamma}<(\beta c_2)n^{1/\beta}\ln
(t{N}^{-1/\gamma}) ,
$$
it follows that
$$
n^{1/\beta}>\frac{(t-1)\cdot {N}^{1/\gamma}}{(\beta c_2)\ln
((t-1){N}^{-1/\gamma})},
$$
and hence the result follows.
\end{proof}

\subsection{Results for Gallai-Ramsey number}

For a fixed probability of receiving colors for each edge, we can
derive the following lower bound of $\operatorname{gr}_k(G,H)$.
\begin{theorem}\label{th-2-4}
Let $G,H$ be two graphs of order $r,s\geq 4$ and size $m_s,m_t$,
respectively. Let $x^*=\min\{x\,|\,{x\choose 2}\geq m_s\}$ and
$y^*=\min\{y\,|\,{y\choose 2}\geq m_t\}$. For $k\geq 2$,
$$
\operatorname{gr}_k(G,H)>\frac{\ell}{e}\cdot \left(k(k-1)(k-2)\cdots
(k-m_s+1)X\left(\frac{1}{k}\right)^{m_s}+k\left(\frac{1}{k}\right)^{m_t}Y\right)^{(-1)/\ell},
$$
where
$$
X={{s\choose 2}\choose m_s}-\sum_{i=1}^{s-x^*}{s\choose
i}{{s-i\choose 2}\choose m_s}, \ Y={{t\choose 2}\choose
m_t}-\sum_{i=1}^{t-y^*}{t\choose i}{{t-i\choose 2}\choose m_t}.
$$
\end{theorem}
\begin{proof}
More precisely, we show that if
$$
\left(\frac{ne}{\ell}\right)^{\ell}N<1,
$$
then ${\rm gr}_k(s,t)>n$, that is, there exists a $k$-coloring of
$K_n$ with vertex set $\{u_1,u_2,\ldots,u_n\}$ containing neither a
rainbow $G$ nor a monochromatic $H$. Consider a random $k$-coloring
of $K_n$, where the color of each edge is determined by the toss of
a fair coin. More precisely, we have a probability space whose
elements are the $k$-colorings of $K_n$, and whose probabilities are
determined by setting
\begin{equation}          \label{eq1-3}
\Pr[\{u_i,u_j\} \ is \ c_x]=k^{-1},
\end{equation}
where $c_1,c_2,\ldots,c_k$ are the all colors and $1\leq x\leq k$,
for all $i,j$ and making these probabilities mutually independent.

Thus there are $k^{{n\choose 2}}$ colorings, each with probability
$k^{-{n\choose 2}}$. For any set of vertices $S$, $|S|=s$, let $A_S$
denote the event that all the induced graphs spanned by $S$ contains
a rainbow $G$. Then
$$
\Pr[A_S]\leq k(k-1)(k-2)\cdots
(k-m_s+1)X\left(\frac{1}{k}\right)^{m_s}.
$$
For any set of vertices $T$, $|T|=t$, let $B_T$ denote the event
that all the induced graphs spanned by $T$ contains a monochromatic
$H$. Then
$$
\Pr[B_T]\leq  k\left(\frac{1}{k}\right)^{m_t}Y.
$$
as the ${t\choose 2}$ ``coin flips'' to determine the colors of
$[S]^2$ must be the same.

The event ``some $S$ of vertices that $G$ is rainbow'' is
represented by $\bigvee_{|S|=s}A_S$, and ``some $T$ of vertices that
$H$ is monochromatic'' is represented by $\bigvee_{|T|=t}B_T$. Then
\begin{eqnarray*}
\Pr\left[\left(\bigvee_{|S|=s}A_S\right)\bigvee
\left(\bigvee_{|T|=t}B_T\right)\right]
&\leq&\Pr\left[\left(\bigvee_{|S|=s}A_S\right)\right]+\Pr\left[\left(\bigvee_{|T|=t}B_T\right)\right]\\[0.2cm]
&\leq &\sum_{|S|=s}\Pr[A_S]+\sum_{|T|=t}\Pr[B_T]\\[0.2cm]
&=&{n\choose s}k(k-1)(k-2)\cdots (k-m_s+1)X\cdot \left(\frac{1}{k}\right)^{m_s}\\[0.2cm]
&&+{n\choose t}k\left(\frac{1}{k}\right)^{m_t}Y\\[0.2cm]
&\leq & {n\choose \ell}N,
\end{eqnarray*}
where $\ell=\max\{\min\{s,n-s\},\min\{t,n-t\}\}$ and
$$
N=k(k-1)(k-2)\cdots
(k-m_s+1)X\left(\frac{1}{k}\right)^{m_s}+k\left(\frac{1}{k}\right)^{m_t}Y.
$$
Furthermore, if
$$
{n\choose \ell}N<\left(\frac{ne}{\ell}\right)^{\ell}N<1,
$$
then
$$
n<\frac{\ell}{eN^{1/\ell}}.
$$
From the above argument, we have
$$
\operatorname{gr}_k(G,H)>\frac{\ell}{e}\cdot \left(k(k-1)(k-2)\cdots
(k-m_s+1)X\left(\frac{1}{k}\right)^{m_s}+k\left(\frac{1}{k}\right)^{m_t}Y\right)^{(-1)/\ell}.
$$
\end{proof}

For two general graphs, we can give lower bound for Gallai-Ramsey
number.
\begin{theorem}\label{th-2-5}
Let $G$ be a graph of order $s\geq 4$ and size $m_s$, respectively.
and $H$ is a complete graph of order $t$. Let $c_1,c_2,c_3$ be three
numbers with $c_3+\frac{c_2}{2}-\frac{c_1c_2^2}{2}<0$. For $k\geq
2$, if and $m_s\geq 2s$, then
$$
\operatorname{gr}_k(G,H)>\left(\frac{(t-1)(s+1)L^{(-1)/(m_s-1)}}{c_2(m_s-1)\ln
\left[{(t-1)L^{(-1)/(m_s-1)}}\right]}\right)^{(m_s-1)/(s+1)},
$$
where
$$
L=m_s(k-1)^{2-m_s}(k-2)(k-3)\cdots \left(k-m_s+1\right)X,
$$
and
$$
X={{s\choose 2}\choose m_s}-\sum_{i=1}^{s-x^*}{s\choose
i}{{s-i\choose 2}\choose m_s}.
$$
\end{theorem}
\begin{proof}
Let the edges of $K_n$ be independently $k$-colored with the
probability that an edge is colored $c_i \ (1\leq i\leq k-1)$ always
being $\frac{p}{k-1}$, and $c_k$ being $1-p$. For a vertex subset
$S$ with exactly $s$ vertices, the event $A_{S}$ that all the
induced graphs spanned by $S$ contains a rainbow $G$. For a vertex
subset $T$ with exactly $t$ vertices, the event $B_{T}$ that all the
induced graphs spanned by $T$ contains a rainbow $H$. Observe that
$\operatorname{gr}_k(G,H)>n$ if
$$
\Pr\left[\left(\bigwedge_S\overline{A_S}\right)\bigwedge
\left(\bigwedge_T\overline{B_T}\right)\right]>0.
$$

Let $\Gamma$ denote the graph ${n\choose s}+{n\choose t}$ vertices
corresponding to all possible $A_S$ and $B_T$, where $\{A_S,B_T\}$
is an edge of $\Gamma$ if and only if $|S\cap T|\geq 2$ and
$|E(G)\cap E(H)|=2$ (i.e., the events $A_S$ and $B_T$ are
independent), the same applies to pairs of the form $\{A_S,A_{S'}\}$
and $\{B_T,B_{T'}\}$. Let $N_{AA}$ denote the number of vertices of
the form $A_S$ for some $S$ joined to some other vertex of this
form, and let $N_{AB},N_{BA}$ and $N_{BB}$ be defined analogously.

If there exist positive $p,y,z$ such that
\begin{equation}          \label{eq4}
\log y>y\Pr[A_S](N_{AA}+1)+z\Pr[B_T]N_{AB},\ \ \ \  \log
z>y\Pr[A_S]N_{BA}+z\Pr[B_T](N_{BB}+1),
\end{equation}
then $gr_k(G,H)>n$. Since
\begin{eqnarray*}
\Pr[A_S]&\leq&(k-1)(k-2)\cdots (k-m_s)X\left(\frac{p}{k-1}\right)^{m_s}\\[0.2cm]
& &+m_s(1-p)(k-1)(k-2)\cdots \left(k-m_s+1\right)X\left(\frac{p}{k-1}\right)^{m_s-1}\\[0.2cm]
&=&m_s(k-1)(k-2)\cdots \left(k-m_s+1\right)X\left(\frac{p}{k-1}\right)^{m_s-1}\cdot \left[(k-m_s)\left(\frac{p}{k-1}\right)+1-p\right]\\[0.2cm]
&\leq&m_s(k-1)(k-2)\cdots \left(k-m_s+1\right)X\left(\frac{p}{k-1}\right)^{m_s-1}\\[0.2cm]
&=&m_s(k-1)^{2-m_s}(k-2)(k-3)\cdots
\left(k-m_s+1\right)Xp^{m_s-1}=Lp^{m_s-1}\\[0.2cm]
\end{eqnarray*}
and
\begin{eqnarray*}
\Pr[B_T]&\leq
&\left[(1-p)^{m_t}+(k-1)\left(\frac{p}{k-1}\right)^{m_t}\right]=
\left[(1-p)^{m_t}+p\left(\frac{p}{k-1}\right)^{m_t-1}\right]\\[0.2cm]
&\leq &
\left[(1-p)^{m_t}+p\left(1-p\right)^{m_t-1}\right]=\left(1-p\right)^{m_t-1}
\end{eqnarray*}
it follows that
$$
N(A,B)={n\choose t}-{n-s\choose t}-s{n-s-1\choose t-1}\leq {n\choose
t}, \ \ N(A,A)+1\leq {n\choose s},
$$
$$
N(B,B)+1={n\choose t}-{n-t\choose t}-t{n-t\choose t-1}<{n\choose t},
\  N(B,A)\leq {n\choose s}.
$$
Set
$$
p=c_1n^{-(s+1)/(m_s-1)}\cdot {L}^{(-1)/(m_s-1)}, \ \ \ \
t-1=c_2n^{(s+1)/(m_s-1)}(\ln n){L}^{1/(m_s-1)},
$$
and
$$
z=\exp\left[c_3n^{(s+1)/(m_s-1)}(\ln n)^2{L}^{1/(m_s-1)}\right], \ \
\ \ y=1+\epsilon.
$$
Observe that
$$
\ln y>y\cdot Lp^{m_s-1}\cdot {n\choose s}+z\cdot
\left(1-p\right)^{m_t-1}\cdot {n\choose t}
$$
and
$$
\ln z>y\cdot Lp^{m_s-1}\cdot {n\choose s}+ z\cdot
\left(1-p\right)^{m_t-1}\cdot {n\choose t}.
$$

Note that $m_t={t\choose 2}$. If
$c_3+\frac{c_2}{2}-\frac{c_1c_2^2}{2}<0$ and $t$ is large, then the
equations (\ref{eq4}) hold. Since
\begin{eqnarray*}
t&=&c_2n^{(s+1)/(m_s-1)}(\ln n){L}^{1/(m_s-1)}+1\\[0.2cm]
&=&c_2n^{(s+1)/(m_s-1)}\frac{m_s-1}{s+1}(\ln n^{(s+1)/(m_s-1)}){L}^{1/(m_s-1)}+1\\[0.2cm]
&\leq&c_2n^{(s+1)/(m_s-1)}\frac{m_s-1}{s+1}(\ln n^{(s+1)/(m_s-1)}){L}^{1/(m_s-1)}+1\\[0.2cm]
&\leq&c_2n^{(s+1)/(m_s-1)}\cdot \frac{m_s-1}{s+1}\cdot \ln
\left[{(t-1)L^{(-1)/(m_s-1)}}\right]\cdot {L}^{1/(m_s-1)}+1,
\end{eqnarray*}
it follows that
$$
n^{(s+1)/(m_s-1)}>\frac{(t-1)(s+1)L^{(-1)/(m_s-1)}}{c_2(m_s-1)\ln
\left[{(t-1)L^{(-1)/(m_s-1)}}\right]},
$$
and hence the result follows.
\end{proof}

\end{document}